\documentclass[a4paper,11pt]{amsart}

\usepackage[
  margin=30mm,
  marginparwidth=25mm,     
  marginparsep=2mm,       
  bottom=25mm,
  ]{geometry}

\usepackage[T1]{fontenc}
\usepackage[utf8]{inputenc}
\usepackage[english]{babel}
\usepackage{amsmath,amssymb,amsthm,mathtools}

\usepackage{latexsym}
\usepackage{delarray}
\usepackage{bbm}
\usepackage{scrextend}
\usepackage{hyperref}
\usepackage[pdftex,usenames,dvipsnames]{xcolor}
\usepackage{datetime}
\usepackage{enumitem}  
\usepackage{tikz}
\usepackage{multicol}
\usepackage{mathrsfs}

\usepackage{comment}

\newtheorem{Th}{Theorem}[section]
\newtheorem{Prop}[Th]{Proposition}
\newtheorem{Lem}[Th]{Lemma}

\newtheorem{Rem}[Th]{Remark}

\newcommand{\C}{\mathbb{C}}
\newcommand{\R}{\mathbb{R}}
\newcommand{\N}{\mathbb{N}}

\newcommand{\vertiii}[1]{{\vert\kern-0.25ex\vert\kern-0.25ex\vert #1 
    \vert\kern-0.25ex\vert\kern-0.25ex\vert}}

\allowdisplaybreaks 
\title[The DNLS in weighted Sobolev spaces]
{Persistence of solutions to the Derivative Nonlinear Schr\"odinger equation in weighted Sobolev spaces}

\author[A. J. Castro]{A. J. Castro}
\author[K. Jabbarkhanov]{K. Jabbarkhanov}
\author[A. Kassimbekov]{A. Kassimbekov}

\address{\newline
       Alejandro J. Castro, Khumoyun Jabbarkhanov,
       Azamat Kassimbekov \newline
       Department of  Mathematics, Nazarbayev University \newline
		010000 Astana, Kazakhstan}
\email{\{alejandro.castilla,khumoyun.jabbarkhanov,azamat.kassimbekov\}@nu.edu.kz}

 \thanks{
This paper was supported by the Ministry of Education and Science of the Republic of Kazakhstan (Grant No AP19676408)}



\begin{document}

\maketitle

\begin{abstract}
In this paper we show the persistence property for solutions of the derivative nonlinear Schr\"odinger equation with initial data in  weighted Sobolev spaces  $H^{2}(\R)\cap L^2(|x|^{2r}dx)$, $r\in (0,1]$. 
\end{abstract}


\section{Introduction}

In this paper we consider the initial value problem for the derivative nonlinear Schr\"odinger (DNLS) equation
\begin{equation}\label{DNLS}
\left\{
\begin{array}{l}
i \partial_t u 
+ \partial_x^2 u 
=i\lambda \partial_x(|u|^2u), \\
u(x,0)=u_{0}(x),
\end{array}
\right.
\end{equation}
which was introduced in connection to the compressible magneto-hydrodynamic equation in the
presence of the Hall effect and the propagation of circular polarized Alfv\'en
waves in magnetized plasma (\cite{MOMT,Mjo}).\\

Local and global well-posedness results for the Cauchy problem 
\eqref{DNLS} have been extensively analysed in Sobolev spaces
(see for instance
\cite{
BaPe,
CKST2001,
CKST2002, 
Hayashi1993TheIV,
Hayashi1992OnTD, HO1994,
Hayashi1994ModifiedWO,
Kaup1978AnES,
Ozawa1996,
Tak1999}),
where the optimal results in $H^{1/2}(\R)$ were deduced by 
H. Takaoka \cite{Tak1999}; 
H. Bahouri and G. Perelman \cite{BaPe}, respectively.\\

Motivated by the seminal work of T. Kato \cite{Kato1983}, the solvability of many dispersive nonlinear equations
 has also been recently studied in the weighted Sobolev spaces $H^{s}(\R)\cap L^2(|x|^{2r}dx)$, aiming to 
control better the decay at infinity of the solutions (see for example 
\cite{
BusJi2018,
FLP2015,
Jim2013,
JenLiuPerSul2018,
MuPa2022,
NP2009}
and the references therein).
In particular in \cite{CN2010, CJZ2022} the authors 
analyzed the behaviour of solutions to the nonlinear Schr\"odinger--Airy equation
\begin{equation*}
\partial_t u 
+ i \, a \, \partial_x^2 u  
+ b \, \partial_x^3 u
+ i \, c \, |u|^2 u 
+ d \, |u|^2 \partial_x u
+ e \, u^2 \partial_x \overline{u} =0
\end{equation*}
in weighted Sobolev spaces, assuming $b \neq 0$ and hence not covering the case of the DNLS. Inspired by the work of T. Ozawa \cite{Ozawa1996}, our main goal is the following persistency property.

\begin{Th}\label{Th:main} 
Let $\lambda \in \R$, $r\in(0,1]$ and 
$u_{0}\in H^{2}(\R)\cap L^2(|x|^{2r}dx)$.
Then, there exist $T>0$ and a unique
solution $u$ of the initial value problem \eqref{DNLS} verifying that
\begin{equation}\label{eq:persistenciofu}
u(\cdot,t)\in H^{2}(\R)\cap L^2(|x|^{2r}dx),
\quad t \in (0,T].    
\end{equation}
\end{Th}

\quad 
\section{Preliminaries}
In this section we introduce the notation and recall some properties that will be often used in this work.\\

We write $A \lesssim B$ when there exists a constant $C>0$ such that $A \leq C B$, and $A \sim B$ if simultaneously $A \lesssim B$ and $B \lesssim A$.
\vspace{1em}

$\mathcal{S}(\R)$ refers to the Schwartz space consisting on smooth functions whose derivatives are rapidly decreasing.
\vspace{1em}

We denote the Fourier transform of a square-integrable function $f$ by
\begin{equation*}
\widehat{f}(\xi)
:= \int_{\R} e^{-ix\xi} f(x) \, dx,
\end{equation*}
which allows to define the Sobolev space $H^s(\mathbb{R})$, of regularity $s \in \mathbb{R}$, via the norm
$$
\|f\|_{H^s}
:=
\Big(  \int_{\R}
(1+|\xi|^2)^{s}
|\widehat{f}(\xi)|^2 \, d\xi\Big)^{1/2}.$$

The inclusion 
\begin{equation}\label{eq:Sobolevinclusion}
\|f\|_{H^s}
\lesssim
\|f\|_{H^{s'}}, \quad s \leq s',
\end{equation}
the product estimate
\begin{equation}\label{eq:Sobolevproduct}
\|fg\|_{H^s}
\lesssim
\|f\|_{H^{s}} \,
\|g\|_{H^{s}}, \quad s > 1/2, 
\end{equation}
and the embedding
\begin{equation}\label{eq:Linfinityembeding}
\|f\|_{L^\infty}
\lesssim 
\|f\|_{H^s}, \quad s > 1/2,
\end{equation}
are standard properties that will be very helpful.
\vspace{1em}

For $\alpha\in\R$, we define the fractional derivative $D_x^{\alpha}$ as the Fourier multiplier given by
$$(D_x^{\alpha}f)^{\wedge}(\xi)
:=
|\xi|^{\alpha}\widehat{f}(\xi).
$$
Analogously, consider
$$
\big((1+D_x^2)^{\alpha}f\big)^{\wedge}(\xi)
:=
(1+|\xi|^2)^{\alpha}\widehat{f}(\xi).
$$
Hence, Plancherel's identity yields to
\begin{equation*}
\|f\|_{H^s}
\sim 
\|(1+D_x^2)^{s/2}f\|_{L^2}
\lesssim 
\|f\|_{L^2}
+
\|D_x^s f\|_{L^2}.
\end{equation*}

On the other hand, the $L^2$-weighted space $L^2(|x|^{2r}dx)$ is 
given by
\begin{equation*}
\|f\|_{L^2(|x|^{2r}dx)}
= 
\||x|^r f\|_{L^2}
=
\Big(  
\int_{\R}
|f(x)|^2 |x|^{2r} \, dx
\Big)^{1/2}.
\end{equation*}
Notice the connection
\begin{equation*}
\|f\|_{L^2(|x|^{2r}dx)}
\sim
\|D^r_{x} \widehat{f}\|_{L^2}.
\end{equation*}

\quad \\
Now, for a function of two variables
 $f : \R \times [0,T] \rightarrow \C$
and $1 \leq p, q < \infty$, let's
define the mixed norm
$$
\|f\|_{L^q_TL^p_x}
:=\Big( 
\int^T_0
\Big(
\int_{\R} |f(x,t)|^p \, dx
\Big)^{q/p} \, dt
\Big)^{1/q},
$$
with the usual modifications involving the essential supremum when  $p=\infty$
or $q=\infty$. It is convenient to introduce
\begin{equation}\label{triplenormdef}
\vertiii{f}
:=\|f\|_{L^{\infty}_T L^2_x}
+
\|f\|_{L^{4}_T L^{\infty}_x},    
\end{equation}
which is a norm recurrently used along this paper.

Next, we turn our attention to the initial value problem for the linear Schr\"odinger equation 
\begin{equation*}
\left\{
\begin{array}{ll}
i \partial_t u 
+ \partial_x^2 u 
=0, & x \in \mathbb{R}, \, t>0, \\
u(x,0)=u_{0}(x), & x \in \R,
\end{array}
\right.
\end{equation*}
whose solution can be written as
$u(x,t)=e^{it \partial_x^2}u_0(x)$, where $e^{it \partial_x^2}$ denotes the Fourier multiplier given by
\begin{equation*}
(e^{it \partial_x^2} u_0)^{\wedge}(\xi)
:= e^{-it\xi^2} \widehat{u_0}(\xi).
\end{equation*}
It is well-known (see e.g. \cite[Th. 4.2]{LP2015}) that  
\begin{equation}\label{property_414}
\Big(\int_{\R}\| e^{it \partial_x^2} f\|^q_{L^p_x} dt\Big)^{1/q}
\lesssim
\|f\|_{L^2}, 
\end{equation}
provided that
\begin{equation}\label{condition_414}
 2\leq p\leq \infty\quad\text{and}\quad 
 \frac{2}{q}=\frac{1}{2}-\frac{1}{p}. 
\end{equation}
Furthermore (\cite[Corollary 4.1]{LP2015}), for any $T>0$,
\begin{equation}\label{corollary_41}
\Big(
\int_0^T
\Big\|\int_0^t e^{i(t-t') \partial_x^2} f(\cdot,t')dt'
\Big\|^{q_1}_{L_x^{p_1}}dt
\Big)^{1/{q_1}}
\lesssim
\Big(
\int_0^T
\|f(\cdot,t')\|^{q'_0}_{L_x^{p'_0}}dt
\Big)^{1/{q'_0}},
\end{equation}
for pairs $(p_0,q_0)$, and $(p_1,q_1)$  satisfying the condition \eqref{condition_414} and with conjugated exponents $p_0'$, $q_0'$
such that
\begin{align*}
\frac{1}{p_0}+\frac{1}{p_0'}=\frac{1}{q_0}+\frac{1}{q_0'}=1.
\end{align*}

\quad \\
In \cite[Lemma 2]{NP2009}  it was also established that, for any
$r\in (0,1)$ and $t>0$,
\begin{equation}\label{semigroup_lemma}
\| |x|^r e^{it \partial_x^2} f \|_{L^2_x}
\lesssim
t^{r/2}
\|f \|_{L^2}
+
t^{r}
\|D_x^{r}f\|_{L^2}
+
\| |x|^r f\|_{L^2}.
\end{equation}

\quad
\section{System of nonlinear Schr\"odinger equations}
Consider the nonlinear Schr\"odinger (NLS) system of equations
\begin{equation}\label{trNSE1}
\left\{
\begin{array}{l}
i\partial_{t}\varphi+\partial_x^2\varphi
= -i\lambda\varphi^2\overline{\psi}, \\
i\partial_{t}\psi+\partial_x^2 \psi
=i\lambda\psi^2\overline{\varphi}, \\
\varphi(x,0)=\varphi_0(x), \\
\psi(x,0)=\psi_0(x),
\end{array}
\right.
\end{equation}
and also set the constraint 
\begin{equation}\label{constraint}
\psi=\partial_x\varphi
+
i\frac{\lambda}{2}|\varphi|^2\varphi.
\end{equation}

The following remark connects the constrained NLS system 
\eqref{trNSE1}-\eqref{constraint} with the DNLS equation \eqref{DNLS} 
(see for example \cite[pp. 147-148]{Ozawa1996}).

\begin{Rem}\label{Lem:phipsiu}
Let $\varphi$ and $\psi$ be solutions of \eqref{trNSE1}  with constraint \eqref{constraint} and
$$\varphi_0(x)
:=\exp\Big(
-i\lambda\int_{-\infty}^{x}|u_{0}(y)|^2dy\Big)u_{0}(x).
$$
Then, if we define
$$u(x,t)
:=\exp\Big(i\lambda\int_{-\infty}^{x}|\varphi(y,t)|^2dy\Big)\varphi(x,t),
$$
it solves the initial value problem \eqref{DNLS}. Conversely, if $u$ is a solution to \eqref{DNLS}, it solves the constrained NLS system \eqref{trNSE1}-\eqref{constraint}
\end{Rem}

\subsection{Unconstrained NLS system}  \quad \\
The next result is a modification of 
\cite[Proposition 2.2]{Ozawa1996}.

\begin{Prop}\label{Prop:system}
Let $\lambda \in \R$, $r \in (0,1]$ and  
$\varphi_0,\psi_0 \in H^2(\R)\cap L^2(|x|^{2r}dx)$. Then, there exist $T>0$ and unique solutions $\varphi$, $\psi$ of the NLS system \eqref{trNSE1}
 verifying that 
\begin{equation}\label{eq:persistenciofvarphipsi}
\varphi(\cdot,t), \psi(\cdot,t)
\in H^{2}(\R)\cap L^2(|x|^{2r}dx),
\quad t \in (0,T].    
\end{equation}
\end{Prop}

\begin{proof} 
We start considering the system of integral equations
\begin{equation}\label{eq:Syst_phi}
\varphi(x,t)
= e^{it \partial_x^2}\varphi_0(x)+i\lambda\int_{0}^te^{i(t-t') \partial_x^2}
(\varphi^2\overline{\psi})(x,t')\, dt',
\end{equation}
\begin{equation}\label{eq:Syst_psi}
\psi(x,t)
=e^{it \partial_x^2}\psi_0(x)+i\lambda\int_{0}^te^{i(t-t') \partial_x^2}
(-\psi^2\overline{\varphi})(x,t')\, dt',    
\end{equation}
which can be combined using vector notation as
\begin{equation}\label{eq:integraleq}
(\varphi,\psi)
=
\Phi(\varphi,\psi),    
\end{equation}
for
 \begin{equation}\label{eq:42}
\Phi(\varphi,\psi):= e^{it \partial_x^2}(\varphi_0,\psi_0)(x)+i\lambda\int_{0}^te^{i(t-t') \partial_x^2}(\varphi^2\overline{\psi},-\psi^2\overline{\varphi})(x,t')\, dt'.  \end{equation}
 
Notice that, by using Duhamel's principle, any solution of the differential equations \eqref{trNSE1} is also a solution of the integral equations \eqref{eq:integraleq}. The converse is not true in general. However, for initial data $\varphi_0,\psi_0 \in H^2(\R)$ one has that a solution of \eqref{eq:integraleq} also satisfies the NLS system \eqref{trNSE1} (proceed as in \cite[Corollary 5.5]{LP2015}).
Therefore, to prove this Proposition we are going to show the existence and uniqueness of solutions of \eqref{eq:integraleq} verifying the persistency property \eqref{eq:persistenciofvarphipsi}.\\

For some constants $a,T>0$, that will be conveniently fixed later, consider the  space
\begin{equation*}\label{eq:43}
X^a_T :=
\{
\varphi \text{ : }
\|\varphi\|_{X_T} \leq a
\},
\end{equation*}
endowed with the norm
\begin{equation*}
\| \varphi \|_{X_T}
:=
 \| \varphi \|_{Y_T}
+\| 
|x|^{r} \varphi\|_{L^\infty_T L^2_x},
\end{equation*}
where
\begin{equation*}
 \| \varphi \|_{Y_T}
:=
\vertiii{\varphi}
+
\vertiii{\partial_x \varphi}
+
\vertiii{\partial^2_x \varphi}
\end{equation*}
with $\vertiii{\cdot}$ introduced in \eqref{triplenormdef}.
Moreover, we define 
\begin{equation*}
\|(\varphi,\psi)\|_{X_T \times X_T}
:=
\|\varphi\|_{X_T}
+
\|\psi\|_{X_T},    
\end{equation*}
and say that $(\varphi,\psi) \in X^a_T\times X^a_T$ provided that 
\begin{equation*}
\|(\varphi,\psi)\|_{X_T \times X_T}
\leq 2a.
\end{equation*}

By the Banach fixed-point theorem, it is enough to establish that the mapping $\Phi$ given in \eqref{eq:42} is a contraction on the complete metric space $X^a_T \times X^a_T$.\\

\underline{\textit{Step 1: $\Phi$ is well defined}}.
Our first goal is to see that $\Phi : X^a_T \times X^a_T \longmapsto X^a_T \times X^a_T$, that is, for $\varphi, \psi \in X^a_T$ we need to show that
\begin{equation}\label{eq:goalPHI}
\|\Phi(\varphi, \psi)\|_{X_T \times X_T}
\leq 2a.
\end{equation}

\quad \\
Recalling well-known results for the unweighted case $r=0$ (see for example
\cite[Proposition 2.2]{Ozawa1996} and the references therein) one can deduce that
\begin{equation}\label{eq:goalPHI'}
\|\Phi(\varphi, \psi)
\|_{Y_{T} \times Y_{T}}
\lesssim 
\|\varphi_0\|_{H^2}
+
\|\psi_0\|_{H^2}
+
T^{\theta}
(
\| \varphi \|_{Y_T}^2 \,
\| \psi \|_{Y_T}
+
\| \varphi \|_{Y_T} \,
\| \psi \|_{Y_T}^2
)
,
\end{equation}
for certain $\theta>0$.
According to 
\eqref{eq:Syst_phi} and \eqref{eq:Syst_psi}, it remains to control the sum of the terms
\vspace{-\baselineskip}
\begin{multicols}{2}
\begin{align*}
& I_1 :=
\| 
|x|^{r}e^{it \partial_x^2}(\varphi_0)
\|_{L^\infty_T L^2_x}, \\
& I_2 :=
\| 
|x|^{r}e^{it \partial_x^2}(\psi_0)
\|_{L^\infty_T L^2_x}, 
\end{align*}
\begin{align*}
& I_3 :=
\Big\| 
|x|^{r}\int_{0}^t
e^{i(t-t') \partial_x^2}(\varphi^2\overline{\psi})(x,t')
\, dt'
\Big\|_{L^\infty_T L^2_x}, \\
& I_4 :=
\Big\| 
|x|^{r}\int_{0}^t
e^{i(t-t') \partial_x^2}(\psi^2\overline{\varphi})(x,t')
\, dt'
\Big\|_{L^\infty_T L^2_x} .
\end{align*}
\end{multicols}
By symmetry, it is enough to consider $I_1$ and $I_3$. 
For the analysis of $I_1$, we simply use
\eqref{semigroup_lemma}, and \eqref{eq:Sobolevinclusion} to get
\begin{align}\label{eq:I1}
I_1 
& \lesssim 
T^{r/2} \| \varphi_0 \|_{L^2}
+
T^{r} \| D^r_x \varphi_0 \|_{L^2}
+
\||x|^r  \varphi_0 \|_{L^2}
\nonumber \\
& \lesssim 
(T^{r/2}+T^{r})\|\varphi_0\|_{H^{1}}
+\||x|^{r}\varphi_0\|_{L^2}.
\end{align}

As for $I_{3}$, once again \eqref{semigroup_lemma} produces
\begin{align}\label{eq:I3}
I_{3}
& \leq
\int_{0}^T
\Big\| 
|x|^{r}
e^{i(t-t') \partial_x^2}(\varphi^2\overline{\psi})(x,t')
\Big\|_{L^\infty_T L^2_x}
\, dt' \nonumber \\
& \lesssim 
T^{1+r/2} \,
\|\varphi^2 \psi\|_{L^\infty_T L^2_x}
+
T^{1+r} \,
\|D_x^{r}(\varphi^2\overline{\psi})\|_{L^\infty_T L^2_x}+
T \, \| |x|^{r} \varphi^2 \psi \|_{L^\infty_T L^2_{x}} \nonumber \\
& \lesssim
(T^{1+r/2} + T^{1+r}) \,
\|\varphi^2 \overline{\psi}\|_{L^\infty_TH^1_{x}}
+
T \,
\|\varphi  \|_{L^\infty_T L_x^{\infty}} \,
\| \psi \|_{L^\infty_T L_x^{\infty}} \,
\||x|^{r}\varphi\|_{L^\infty_T  L_x^2}
\nonumber \\
& \lesssim
(T^{1+r/2} + T^{1+r}) \,
 \,
\|\varphi\|_{L^\infty_TH^1_{x}}^2
\|\psi\|_{L^\infty_TH^1_{x}} 
+
T \,
\|\varphi  \|_{L^\infty_T H_x^{1}} \,
\| \psi \|_{L^\infty_T H_x^{1}} \,
\||x|^{r}\varphi\|_{L^\infty_T  L_x^2} \nonumber \\
& \lesssim 
(T^{1+r/2} + T^{1+r} + T)
\|\varphi \|_{X_T}^2 \,
\| \psi \|_{X_T}
\end{align}
where we also used the Sobolev properties 
\eqref{eq:Sobolevproduct}
and
\eqref{eq:Linfinityembeding}.\\

Combining 
\eqref{eq:goalPHI'},
\eqref{eq:I1} and 
\eqref{eq:I3}
(and taking  into consideration the analogous estimates for $I_2$ and $I_4$) we deduce that,
for every $\varphi, \psi \in X^a_T$,
\begin{align*}
\|\Phi(\varphi, \psi)\|_{X_T \times X_T}
& \leq 
 C
\Big(
(1+T^{\theta})
(\|\varphi_0\|_{H^{2}}
+
\|\psi_0\|_{H^{2}})
+
\||x|^{r}\varphi_0\|_{L^2}
+
\||x|^{r}\psi_0\|_{L^2}
\Big) \\
& \quad + 
C \,
T^{\theta} \, a^3,
\end{align*}
for some $C, \theta>0$.\\

Finally, set
\begin{equation*}
a:= 
C
(
\|\varphi_0\|_{H^{2}}
+
\|\psi_0\|_{H^{2}}
+
\||x|^{r}\varphi_0\|_{L^2}
+
\||x|^{r}\psi_0\|_{L^2}
)
\end{equation*}
and chose $T$ small enough verifying
\begin{equation}\label{eq:conditionforT}
 C \, T^{\theta} \,
(\|\varphi_0\|_{H^{2}}
+
\|\psi_0\|_{H^{2}})
+
C \, T^{\theta} \,
 a^3
\leq a,
\end{equation}
to conclude \eqref{eq:goalPHI}.\\

\underline{\textit{Step 2: $\Phi$ is a contraction}}.
Now the task is to show the existence of a constant $0<K<1$
verifying that
\begin{equation*}
\|\Phi(\varphi, \psi)
- \Phi(\widetilde{\varphi}, \widetilde{\psi})\|_{X_T \times X_T}
\leq K \|(\varphi, \psi)
-
(\widetilde{\varphi}, \widetilde{\psi})\|_{X_T\times X_T},
\end{equation*}
for $\varphi, \widetilde{\varphi},
\psi, \widetilde{\psi} \in X^a_T$.

It is known that (see
\cite[Proposition 2.2]{Ozawa1996}),
for some $\theta >0$,
\begin{align}\label{eq:KY_t}
\|\Phi(\varphi, \psi)
- \Phi(\widetilde{\varphi}, \widetilde{\psi})\|_{Y_T \times Y_T}
& \lesssim 
T^\theta 
\big(
\|\varphi\|_{Y_T}
+\|\widetilde{\varphi}\|_{Y_T}
+\|\psi\|_{Y_T}
+\|\widetilde{\psi}\|_{Y_T}
\big)^2 \nonumber \\
& \quad \times
\|(\varphi, \psi)
-
(\widetilde{\varphi}, \widetilde{\psi})\|_{Y_T \times Y_T}.
\end{align}
Hence, we concentrate on the terms
\begin{align*}
J_1
&:=\Big\| |x|^{r}\int_{0}^te^{i(t-t') \partial_x^2}
(\varphi^2\overline{\psi}-\widetilde{\varphi}^2\overline{\widetilde{\psi}})(x,t')\, dt'\Big\|_{L^\infty_T L^2_x}, \\
J_2
&:=\Big\| |x|^{r}\int_{0}^te^{i(t-t') \partial_x^2}
(\psi^2\overline{\varphi}-\widetilde{\psi}^2\overline{\widetilde{\varphi}})(x,t')\, dt'\Big\|_{L^\infty_T L^2_x}.
\end{align*}
We only analyze $J_1$, $J_2$ can be treated analogously.\\

Next, observe that we can rewrite the difference of nonlinear factors as follows
\begin{equation}\label{eq:34}
\varphi^2
\overline{\psi}
-
\widetilde{\varphi}^2
\overline{\widetilde{\psi}}
=
\varphi^2
(\overline{\psi}
-
\overline{\widetilde{\psi}})
+
\overline{\widetilde{\psi}}
(\varphi+
\widetilde{\varphi})
(\varphi-
\widetilde{\varphi}).
\end{equation}

\quad \\
Then, proceeding as in 
\eqref{eq:I3}
and
\eqref{eq:KY_t}
above, we get
\begin{align}\label{eq:J1}
J_1
& \lesssim
T \,
\| |x|^{r} 
(\varphi^2\overline{\psi}-\widetilde{\varphi}^2\overline{\widetilde{\psi}}) \|_{L^\infty_T L^2_{x}}  +
(T^{1+r/2}+T^{1+r}) \,
\| \varphi^2\overline{\psi}-\widetilde{\varphi}^2\overline{\widetilde{\psi}} \|_{L^\infty_T H^2_x} \nonumber \\
& \leq
T \,
\Big(
\| |x|^{r} \varphi\|_{L^\infty_T L^2_{x}} \,
\| \varphi \|_{L^\infty_T H^2_{x}} \,
\| \psi-\widetilde{\psi} \|_{L^\infty_T H^2_{x}} \nonumber \\
& \qquad +
\| |x|^{r} \widetilde{\psi}\|_{L^\infty_T L^2_{x}} \,
\| \varphi \|_{L^\infty_T H^2_{x}} \,
\| \varphi-\widetilde{\varphi}\|_{L^\infty_T H^2_{x}} \nonumber \\
& \qquad +
\| |x|^{r} \widetilde{\psi}\|_{L^\infty_T L^2_{x}} \,
\| \widetilde{\varphi} \|_{L^\infty_T H^2_{x}} \,
\| \varphi-\widetilde{\varphi}\|_{L^\infty_T H^2_{x}}
\Big) \nonumber \\
& \quad +
(T^{1+r/2}+T^{1+r})
\Big(
\|\varphi\|_{L^\infty_T H^2_{x}}^2 \,
\| \psi-\widetilde{\psi} \|_{L^\infty_T H^2_{x}} 
+
\|\widetilde{\psi}\|_{L^\infty_T H^2_{x}} \,
\|\varphi\|_{L^\infty_T H^2_{x}} \,
\| \varphi-\widetilde{\varphi} \|_{L^\infty_T H^2_{x}} \nonumber \\
& \hspace{4cm} +
\|\widetilde{\psi}\|_{L^\infty_T H^2_{x}} \,
\|\widetilde{\varphi}\|_{L^\infty_T H^2_{x}} \,
\| \varphi-\widetilde{\varphi} \|_{L^\infty_T H^2_{x}} 
\Big).
\end{align}

\quad \\
Putting together \eqref{eq:KY_t}, 
\eqref{eq:J1} and the corresponding symmetric estimate for $J_2$ we get, for $\varphi, \widetilde{\varphi},
\psi, \widetilde{\psi} \in X^a_T$,
\begin{align*}
\|\Phi(\varphi, \psi)
- \Phi(\widetilde{\varphi}, \widetilde{\psi})\|_{X_T \times X_T}
& \leq 
C T^\theta a^2  
 \|(\varphi, \psi)
-
(\widetilde{\varphi}, \widetilde{\psi})\|_{X_T\times X_T},
\end{align*}
for certain $C,\theta>0$.
Therefore, it is possible to pick $T$ satisfying simultaneously \eqref{eq:conditionforT} and $C T^\theta a^2 <1$.
\end{proof}

\subsection{Constrained NLS system} \quad \\
The following lemma can be shown as in \cite[Proposition 2.3]{Ozawa1996} (see also \cite[Proposition 2.4]{HO1994}), since our extra hypothesis concerning the weighted Sobolev initial data do not play a crucial role in the proofs.

\begin{Lem}\label{Prop:u2}
Let 
$\varphi_0, \psi_0 \in \mathcal{S}(\R)$ related by the constraint
\begin{equation}\label{eq:constrain_psi0}
\psi_0
=\partial_x \varphi_0+i\frac{\lambda}{2}|\varphi_0|^2\varphi_0.    
\end{equation}
For $T>0$ sufficiently small and $\varphi$, $\psi$ the solutions
of the NLS system \eqref{trNSE1} provided by  Proposition \ref{Prop:system}, we have that 
\begin{equation}\label{eq:constrainforallt}
\psi(\cdot,t)
=\partial_x\varphi (\cdot,t)
+i\frac{\lambda}{2}|\varphi|^2\varphi(\cdot,t), \quad t \in (0,T].
\end{equation}
\end{Lem}

\begin{Prop}\label{Prop:u}
Let $r \in (0,1]$ and $\rho>0$.
Assume also that
$\varphi_0, \psi_0 \in \mathcal{S}(\R)$ 
verify
\begin{equation*}
\psi_0
=\partial_x \varphi_0+i\frac{\lambda}{2}|\varphi_0|^2\varphi_0    
\end{equation*}
and
\begin{equation*}
\|\varphi_0\|_{H^1}
+\|\psi_0\|_{H^1}
+\||x|^{r}\varphi_0\|_{L^2}
\leq \rho.
\end{equation*}
Then, there exists $T>0$ such that the corresponding solutions $\varphi,\psi$ 
of the NLS system \eqref{trNSE1} (given by  Proposition \ref{Prop:system}) satisfy
\begin{equation}\label{eq:(52)}
\vertiii{\varphi}
+
\vertiii{\partial_x\varphi}
+
\vertiii{\partial_x^2 \varphi}
+
\vertiii{\psi}
+
\vertiii{\partial_x\psi}
+
\||x|^{r}\varphi\|_{L^\infty_T L^2_x}
\leq C(\rho, T),
\end{equation}
where $C(\rho, T)$ denotes a positive constant depending only on $\rho$ and $T$.
Furthermore, for any other pair of initial data
$\widetilde{\varphi}_0, \widetilde{\psi}_0 \in \mathcal{S}(\R)$ 
verifying  also
$$\widetilde{\psi}_0
=\partial_x\widetilde{\varphi}_0+i\frac{\lambda}{2}|\widetilde{\varphi}_0|^2\widetilde{\varphi}_0
$$
and
\begin{equation*}
\|\widetilde{\varphi}_0\|_{H^1}
+\|\widetilde{\psi}_0\|_{H^1}
+\||x|^{r}\widetilde{\varphi}_0\|_{L^2}
\leq \rho
\end{equation*}
with associated solutions $\widetilde{\varphi}$ and $\widetilde{\psi}$, we have that 
\begin{align}\label{eq:37}
& \vertiii{\varphi-\widetilde{\varphi}}
+\vertiii{\partial_x(\varphi-\widetilde{\varphi})}
+\vertiii{\partial_x^2(\varphi-\widetilde{\varphi})} 
\nonumber \\ 
& \qquad 
+\vertiii{\psi-\widetilde{\psi}}
+\vertiii{\partial_x (\psi- \widetilde{\psi})}
+ \||x|^{r}(\varphi-\widetilde{\varphi})\|_{L^\infty_T L^2_x}
\nonumber\\
&\quad  \lesssim
\|\varphi_0-\widetilde{\varphi}_0\|_{H^1}
+\|\psi_0-
\widetilde{\psi}_0\|_{H^1}
+\||x|^{r}
(\varphi_0-
\widetilde{\varphi}_0)\|_{L^2},
\end{align}
where the hidden constant in \eqref{eq:37} might depend on $\rho$ and $T$.
\end{Prop}

We present the proofs of 
\eqref{eq:(52)}
and
\eqref{eq:37}
separately below. 

\begin{proof}[Proof of \eqref{eq:(52)}]

From \cite[Proposition 2.4]{Ozawa1996}), we can find $T_1>0$ satisfying that 
\begin{equation}\label{eq:41}
\vertiii{\varphi} 
+
\vertiii{\partial_x\varphi}
+
\vertiii{\psi} 
\leq C(\rho,T_1).
\end{equation}
Therefore, it remains to guarantee that 
\begin{equation}\label{eq:ABC}
\mathcal{A} 
+ \mathcal{B} 
+ \mathcal{C} 
:=\vertiii{ \partial_x \psi} +\vertiii{\partial_x^2\varphi}+ \||x|^{r}\varphi\|_{L^\infty_T L^2_x}
<C(\rho,T),
\end{equation}
for some $T \in (0,T_1]$ sufficiently small.\\

Since, $\psi$ solves the integral equation \eqref{eq:Syst_psi} (recall the remark at the begining of the proof of Proposition \ref{Prop:system}), we have that
\begin{align}\label{eq:Adecomp}
\mathcal{A}
&\lesssim
\| e^{it \partial_x^2}(\partial_x\psi_0)\|_{L_T^{\infty}L^2_x}
+
\| e^{it \partial_x^2}
(\partial_x \psi_0)\|_{L_T^{4}L^\infty_x} \nonumber \\
& \qquad +
\Big\|\int_{0}^te^{i(t-t') \partial_x^2}
\partial_x (\psi^2\overline{\varphi})(x,t')\, dt' \Big\|_{L_T^{\infty}L^2_x} \nonumber \\
& \qquad +
\Big\|\int_{0}^te^{i(t-t') \partial_x^2}
\partial_x(\psi^2\overline{\varphi})(x,t')\, dt' \Big\|_{L_T^{4}L^\infty_x} \nonumber \\
& =: 
\mathcal{A}_1
+ \mathcal{A}_2
+ \mathcal{A}_3
+ \mathcal{A}_4.
\end{align}
From, \eqref{property_414} we immediately obtain  
\begin{equation}\label{eq:A1A2}
\mathcal{A}_1 + \mathcal{A}_2
\lesssim 
\|\partial_x\psi_0\|_{L^2}
\lesssim 
\|\psi_0\|_{H^1} < \rho.
\end{equation}
Similarly, using
\eqref{corollary_41}
and
\eqref{eq:Sobolevproduct} we get
\begin{align}\label{eq:A3A4}
\mathcal{A}_3 + \mathcal{A}_4
&\lesssim 
\|\partial_x(\psi^2\overline{\varphi})\|_{L^1_T L^2_x}
\lesssim
T
\|\psi^2\overline{\varphi}\|_{L^\infty_T H^1_x}
\lesssim
T
\|\varphi\|_{L^\infty_T H^1_x}
\|\psi \|_{L^\infty_T H^1_x}^2 \nonumber \\
& \sim T
\Big(
\|\varphi\|_{L^\infty_T L^2_x} 
+ \|\partial_x \varphi\|_{L^\infty_T L^2_x}
\Big)
\Big( 
\|\psi\|_{L^\infty_T L^2_x}^2
+ \|\partial_x \psi\|_{L^\infty_T L^2_x}^2
\Big) \nonumber \\
& \lesssim T 
\Big(
\vertiii{\varphi}
+ 
\vertiii{\partial_x \varphi}
\Big)
\Big( 
\vertiii{\psi}^2
+ \vertiii{\partial_x \psi}^2
\Big).
\end{align}
Combining 
\eqref{eq:41}, \eqref{eq:A1A2} and \eqref{eq:A3A4} 
we deduce, for any $T<T_1$,
\begin{equation*}
\mathcal{A}
\lesssim 
\rho
+
T \,
C(\rho,T_1)
\Big( 
C(\rho,T_1)^2
+ \mathcal{A}^2
\Big).
\end{equation*}
Hence, it is possible to choose $T_2<T_1$ sufficiently small verifying that
\begin{equation}\label{eq:A}
\mathcal{A}
\lesssim C(\rho,T_2).  
\end{equation}

\quad \\
For the analysis of the factor $\mathcal{B}$, we  apply \eqref{eq:constrainforallt} to write
\begin{equation}\label{eq:B12}
\mathcal{B}
\lesssim 
\vertiii{\partial_x\psi}
+
\vertiii{\partial_x(|\varphi|^2 \varphi)}
=: \mathcal{B}_1 + \mathcal{B}_2.
\end{equation}
Notice that $\mathcal{B}_1=\mathcal{A}$, treated above. On the other hand, \eqref{eq:Linfinityembeding} yields to
\begin{align}\label{eq:B2}
\mathcal{B}_2
& \lesssim
\|
\varphi^2
\partial_x \overline{\varphi}
\|_{L^{\infty}_T L^2_x}
+
\|
\varphi^2
\partial_x \overline{\varphi}
\|_{L^{4}_TL^{\infty}_x}
+
\||\varphi|^2
\partial_x \varphi\|_{L^{\infty}_T L^2_x}
+
\| |\varphi|^2 
\partial_x \varphi\|_{L^{4}_TL^{\infty}_x} \nonumber \\
& \lesssim 
\|\partial_x \varphi\|_{L^{\infty}_T L^2_x}
\|\varphi\|^2_{L^{\infty}_T H^1_x}
+\|\partial_x \varphi\|_{L^{4}_T L^\infty_x}
\|\varphi\|^2_{L^{\infty}_T H^1_x} \nonumber \\
& \lesssim 
\vertiii{\partial_x \varphi}
\Big(
\vertiii{\varphi}^2 
+
\vertiii{\partial_x\varphi}^2
\Big).
\end{align}
Thus, \eqref{eq:41} and \eqref{eq:A} imply
\begin{equation}\label{eq:B}
\mathcal{B}
\lesssim C(\rho,T_2).  
\end{equation}

\quad \\
We finally consider the term 
$\mathcal{C}$. Using this time the integral equation \eqref{eq:Syst_phi}, one realizes that
\begin{align}\label{eq:CC}
\mathcal{C}
& \lesssim 
\||x|^{r} e^{it \partial_x^2}\varphi_0 \|_{L^\infty_T L^2_x}
+
\Big\||x|^{r} \int_{0}^te^{i(t-t') \partial_x^2}
(\varphi^2 \overline{\psi})
(x,t')\, dt' \Big\|_{L^\infty_T L^2_x}
=: \mathcal{C}_1+ \mathcal{C}_2.
\end{align}
Next, similarly to \eqref{eq:I1} we deduce, for any $T<T_2$,
\begin{align}\label{eq:C1}
\mathcal{C}_1 
&\lesssim
(T^{r/2}+T^m)
\|\varphi_0\|_{H^1}
+\||x|^{r}\varphi_0\|_{L^2}
 \lesssim 
C(\rho, T_2).
\end{align}
As for $\mathcal{C}_2$ we proceed as in \eqref{eq:I3} above to get
\begin{align}\label{eq:C2}
\mathcal{C}_2
&\lesssim 
T \,
\|\varphi  \|_{L^\infty_T H_x^{1}} \,
\| \psi \|_{L^\infty_T H_x^{1}} \,
\||x|^{r}\varphi\|_{L^\infty_T  L_x^2}
+
(T^{1+r/2} + T^{1+r}) \,
 \,
\|\varphi\|_{L^\infty_TH^1_{x}}^2
\|\psi\|_{L^\infty_TH^1_{x}}
\nonumber \\
& \lesssim
T \,
\Big( 
\vertiii{\varphi}
+ 
\vertiii{\partial_x \varphi}
\Big)
\Big( 
\vertiii{\psi}
+ 
\vertiii{\partial_x \psi}
\Big)
\||x|^{r}\varphi\|_{L^\infty_T  L_x^2} \nonumber \\
& \qquad +
(T^{1+r/2} + T^{1+r}) \,
\Big( 
\vertiii{\varphi}^2
+ 
\vertiii{\partial_x \varphi}^2
\Big)
\Big( 
\vertiii{\psi}
+ 
\vertiii{\partial_x \psi}
\Big).
\end{align}
Therefore, putting together \eqref{eq:41}, \eqref{eq:A},
\eqref{eq:C1} and \eqref{eq:C2}, one sees that for $T<T_2$
\begin{align*}
\mathcal{C}
\lesssim 
C(\rho,T_2)
+ T C(\rho,T_2) \mathcal{C}.
\end{align*}
Thus, taking $T_3<T_2$ small enough, we conclude
\begin{equation}\label{eq:C}
\mathcal{C}
\lesssim C(\rho,T_3).
\end{equation}

\quad\\
In summary, our goal \eqref{eq:ABC} is a consequence of \eqref{eq:A}, \eqref{eq:B} and \eqref{eq:C}.
\end{proof}

\begin{proof}[Proof of \eqref{eq:37}]
In most inequalities below the hidden constants might depend on $\rho$ and $T$ as in \eqref{eq:(52)}, however for simplicity we are not going to always specify them explicitly. 
Recall from \cite[Proposition 2.4]{Ozawa1996} that
\begin{align}\label{ozawaspart}
\vertiii{\varphi-\widetilde{\varphi}}
+
\vertiii{\partial_x\varphi-\partial_x\widetilde{\varphi}}
+\vertiii{\psi-\widetilde{\psi}}
& \lesssim
\|\varphi_0-\widetilde{\varphi}_0\|_{L^2}
+\|\psi_0-\widetilde{\psi}_0\|_{L^2}.
\end{align}
Then, it is sufficient to establish
\begin{align}\label{eq:goal'}
\mathscr{A} 
+ \mathscr{B} 
+ \mathscr{C} 
&:=
\vertiii{\partial_x\psi-\partial_x\widetilde{\psi}}
+
\vertiii{\partial^2_x\varphi-\partial^2_x\widetilde{\varphi}}
+ 
\||x|^{r}(\varphi-\widetilde{\varphi})\|_{L^\infty_T L^2_x} \nonumber \\
&\lesssim
\|\varphi_0-\widetilde{\varphi}_0\|_{H^1}
+\|\psi_0-\widetilde{\psi}_0\|_{H^1}
+\||x|^{r}(\varphi_0-\widetilde{\varphi}_0)\|_{L^2}.    
\end{align}

\quad \\
We start controlling $\mathscr{A}$. 
Similarly to steps
\eqref{eq:Adecomp},
\eqref{eq:A1A2}
and
\eqref{eq:A3A4}
above and using the key observation 
\begin{equation*}
\psi^2\overline{\varphi}
-
(\widetilde{\psi})^2
\overline{\widetilde{\varphi}}
=
\psi^2
\overline{
(\varphi-\widetilde{\varphi})}
+
\overline{\widetilde{\varphi}}
(\psi+\widetilde{\psi})
(\psi-\widetilde{\psi}),
\end{equation*}
it is not difficult to deduce
\begin{align*}
\mathscr{A}
&\lesssim
\|\psi_0 - \widetilde{\psi}_0\|_{H^1}
+ T
\|
\psi^2\overline{\varphi}
-
(\widetilde{\psi})^2
\overline{\widetilde{\varphi}}
\|_{L^\infty_T H^1_x} \nonumber \\  
& \lesssim
\|\psi_0 - \widetilde{\psi}_0\|_{H^1}
+ T
\|
\psi^2
\overline{
(\varphi-\widetilde{\varphi})}
\|_{L^\infty_T H^1_x}
+ T
\|
\overline{\widetilde{\varphi}}
(\psi+\widetilde{\psi})
(\psi-\widetilde{\psi})
\|_{L^\infty_T H^1_x} \nonumber \\
& \lesssim
\|\psi_0 - \widetilde{\psi}_0\|_{H^1}
+ T
\Big(
\vertiii{\varphi-\widetilde{\varphi}}
+ 
\vertiii{\partial_x \varphi
- \partial_x \widetilde{\varphi}}
\Big)
\Big( 
\vertiii{\psi}^2
+ \vertiii{\partial_x \psi}^2
\Big) \nonumber \\
& \qquad + T
\Big( 
\vertiii{\widetilde{\varphi}}
+ \vertiii{\partial_x \widetilde{\varphi}}
\Big)
\Big( 
\vertiii{\psi}
+ \vertiii{\partial_x \psi}
+ \vertiii{\widetilde{\psi}}
+ \vertiii{\partial_x \widetilde{\psi}}
\Big) \nonumber \\
& \qquad \qquad \times
\Big(
\vertiii{\psi-\widetilde{\psi}}
+ 
\vertiii{\partial_x \psi
- \partial_x \widetilde{\psi}}
\Big).
\end{align*}
Furthermore, 
\eqref{eq:(52)}
and
\eqref{ozawaspart} yield to
\begin{equation*}
\mathscr{A}
\lesssim    
\|\varphi_0-\widetilde{\varphi}_0\|_{L^2}
+\|\psi_0-\widetilde{\psi}_0\|_{H^1}
+ T \mathscr{A},
\end{equation*}
which allows us to take $T>0$ sufficiently small to conclude
\begin{equation}\label{eq:A''}
\mathscr{A}
\lesssim    
\|\varphi_0-\widetilde{\varphi}_0\|_{L^2}
+\|\psi_0-\widetilde{\psi}_0\|_{H^1}.
\end{equation}

\quad \\
We decompose $\mathscr{B}$ as in 
\eqref{eq:B12},
\begin{align*}
\mathscr{B}
& \lesssim 
\vertiii{
\partial_x\psi
-
\partial_x\widetilde{\psi}}
+
\vertiii{\partial_x(
|\varphi|^2 \varphi
-
|\widetilde{\varphi}|^2 \widetilde{\varphi})} 
=: \mathscr{B}_1 + \mathscr{B}_2
\end{align*}
and observe that $\mathscr{B}_1=\mathscr{A}$.
On the other hand, applying this time the relations
\begin{align*}
\varphi^2
\partial_x \overline{\varphi}
-
\widetilde{\varphi}^2
\partial_x \overline{\widetilde{\varphi}}
& = 
\varphi^2 
\overline{(\partial_x \varphi
-
\partial_x \widetilde{\varphi})}
+
\partial_x \overline{\widetilde{\varphi}}
(\varphi + \widetilde{\varphi})
(\varphi-\widetilde{\varphi})
\\
|\varphi|^2 \partial_x \varphi
-
|\widetilde{\varphi}|^2
\partial_x \widetilde{\varphi}
& =
|\varphi|^2 
(\partial_x \varphi - \partial_x \widetilde{\varphi})
+
\partial_x \widetilde{\varphi}
\, \varphi \, 
\overline{(\varphi-\widetilde{\varphi})}
+
\partial_x \widetilde{\varphi}
\, \overline{\widetilde{\varphi}} \, 
(\varphi-\widetilde{\varphi})
\end{align*}
and reasoning as in \eqref{eq:B2}, give
\begin{align*}
\mathscr{B}_2
& \lesssim
\vertiii{
\varphi^2\partial_x \overline{\varphi}
-
\widetilde{\varphi}^2
\partial_x \overline{\widetilde{\varphi}} 
}
+
\vertiii{
|\varphi|^2 \partial_x \varphi
-
|\widetilde{\varphi}|^2
\partial_x \widetilde{\varphi}
} \\
& \lesssim 
\Big( 
\vertiii{\varphi}
+
\vertiii{\widetilde{\varphi}}
+
\vertiii{\partial_x \varphi}
+
\vertiii{\partial_x \widetilde{\varphi}}
+
\vertiii{\partial_x^2 \widetilde{\varphi}}
\Big)^2 \\
& \qquad \times \Big( 
\vertiii{\varphi - \widetilde{\varphi}}
+
\vertiii{\partial_x \varphi - \partial_x \widetilde{\varphi}}
\Big).
\end{align*}
Thus, taking into account
\eqref{eq:(52)} and
\eqref{ozawaspart}
\begin{equation}\label{eq:B'}
\mathscr{B}
\lesssim
\|\varphi_0-\widetilde{\varphi}_0\|_{L^2}
+\|\psi_0-\widetilde{\psi}_0\|_{H^1}.
\end{equation}

\quad \\
Finally, by repeating the arguments from \eqref{eq:CC}, we have
\begin{align*}
\mathscr{C}
&\lesssim 
\||x|^{r} 
e^{it \partial_x^2}
(\varphi_0-\widetilde{\varphi}_0) \|_{L^\infty_T L^2_x}
+
\Big\||x|^{r} \int_{0}^te^{i(t-t') \partial_x^2}(\varphi^2\overline{\psi}-\widetilde{\varphi}^2
\overline{\widetilde{\psi}})(x,t')\, dt' \Big\|_{L^\infty_T L^2_x}\\
&:=
\mathscr{C}_1+\mathscr{C}_2,
\end{align*}
and again, as in \eqref{eq:C1}, it follows 
\begin{equation*}
\mathscr{C}_1
\lesssim
\|\varphi_0-\widetilde{\varphi}_0\|_{H^1}
+\||x|^{r}(\varphi_0-\widetilde{\varphi}_0)\|_{L^2}.
\end{equation*}
The treatment of $\mathscr{C}_2$ is parallel to \eqref{eq:C2}, after rewriting the difference of nonlinear terms via
\eqref{eq:34}, namely
\begin{align*}
\mathscr{C}_2    
& \lesssim
\Big(
\|\varphi  \|_{L^\infty_T H_x^{1}} \,
\||x|^{r}\varphi\|_{L^\infty_T  L_x^2}  \,
+
\|\varphi\|_{L^\infty_TH^1_{x}}^2
\Big) 
\| \psi - \widetilde{\psi} \|_{L^\infty_T H_x^{1}}\\
& \quad +
\Big(
\|\varphi\|_{L^\infty_TH^1_{x}}  
+
\|\widetilde{\varphi}\|_{L^\infty_TH^1_{x}}
+
\||x|^{r}\varphi\|_{L^\infty_T  L_x^2}  
+
\||x|^{r}\widetilde{\varphi}\|_{L^\infty_T  L_x^2}  
\Big) 
\|\widetilde{\psi}  \|_{L^\infty_T H_x^{1}} \,
\|\varphi - \widetilde{\varphi}\|_{L^\infty_TH^1_{x}}.
\end{align*}
As usual, from here 
\eqref{eq:(52)}, \eqref{ozawaspart}
and \eqref{eq:A''}
lead to
\begin{equation*}
\mathscr{C}_2
\lesssim
\|\varphi_0-\widetilde{\varphi}_0\|_{L^2}
+\|\psi_0-\widetilde{\psi}_0\|_{H^1},
\end{equation*}
so
\begin{equation}\label{eq:C'}
\mathscr{C}
\lesssim    
\|\varphi_0-\widetilde{\varphi}_0\|_{H^1}
+\|\psi_0-\widetilde{\psi}_0\|_{H^1}
+\||x|^{r}(\varphi_0-\widetilde{\varphi}_0)\|_{L^2}.
\end{equation}

\quad \\
In brief, \eqref{eq:A''}, \eqref{eq:B'} and \eqref{eq:C'} show \eqref{eq:goal'}.
\end{proof}

\quad
\section{Proof of Theorem \ref{Th:main}} 

We start defining the initial data
$$\varphi_0(x)
:=\exp\Big(
-i\lambda\int_{-\infty}^{x}|u_{0}(y)|^2dy\Big)u_{0}(x).
$$
Observe that $\varphi_0 \in  H^{2}(\R)$. However, 
if we would simply take
$$\psi_0
:=\partial_x \varphi_0+i\frac{\lambda}{2}|\varphi_0|^2\varphi_0,
$$
we could not guarantee that $\psi_0 \in H^{2}(\R)$ and hence Proposition \ref{Prop:system} does not apply for $\varphi_0, \psi_0$. 
Instead, we consider a sequence $\{\varphi_0^{(n)}\}_{n \in \mathbb{N}} \subset \mathcal{S}(\R)$ converging to $\varphi_0$ in $H^{2}(\R) \cap L^2(|x|^{2r} dx)$, that is,
\begin{equation}\label{limitdifference_of_varphi}
\lim_{n \to \infty} 
\Big(
\|\varphi^{(n)}_{0} - \varphi_{0}\|_{H^{2}(\R)}
+
\||x|^{r}(\varphi^{(n)}_{0} - \varphi_{0})\|_{L^2(\R)} \Big)
=0.
\end{equation}
Moreover, let
\begin{equation}\label{eq:psi0(n)}
\psi_{0}^{(n)}(x)
:=
\partial_x\varphi_0^{(n)}(x)+i\frac{\lambda}{2}|\varphi_0^{(n)}(x)|^2\varphi^{(n)}_0(x), 
\qquad n \in \mathbb{N},    
\end{equation}
which are again regular functions in $\mathcal{S}(\R) \subset H^{2}(\R) \cap L^2(|x|^{2r} dx)$.

By \eqref{limitdifference_of_varphi} we can choose $\rho>0$ such that 
\begin{equation}\label{eq:rho(n)}
\|\varphi_0^{(n)}\|_{H^1}
+
\|\psi_0^{(n)}\|_{H^1}
+
\||x|^{r}\varphi_0^{(n)}\|_{L^2}
\leq \rho, \quad n \in \N.
\end{equation}

Next, for each $n \in \N$, denote by $(\varphi^{(n)}, \psi^{(n)})$ the solution of the NLS system \eqref{trNSE1} provided by Proposition \ref{Prop:system} corresponding to the initial data $(\varphi^{(n)}_0, \psi^{(n)}_0)$. Additionally introduce
\begin{equation*}
u^{(n)}(x,t)
:= 
\exp\Big(
i\lambda\int_{-\infty}^{x}|\varphi^{(n)}(y,t)|^2dy\Big)
\varphi^{(n)}(x,t)
\end{equation*}
and
\begin{equation*}
u_0^{(n)}(x)
:= 
\exp\Big(
i\lambda\int_{-\infty}^{x}|\varphi^{(n)}_0(y)|^2dy\Big)
\varphi^{(n)}_0(x).
\end{equation*}

Our aim is to see that the sequence $\{u^{(n)}\}_{n \in \N}$ converges to certain function $u$ which satisfies \eqref{DNLS} and \eqref{eq:persistenciofu}. The following technical estimate will be the key.

\begin{Lem}
Fix $r\in(0,1]$ and  let 
$\{\varphi_0^{(n)}\}_{n \in \mathbb{N}}$,
$\{\psi_0^{(n)}\}_{n \in \mathbb{N}}$
and
$\{u^{(n)}\}_{n \in \mathbb{N}}$ be the sequences defined above. 
Then, there exist $\rho, T>0$ verifying that,
for any $m,n \in \N$,
\begin{equation}\label{eq:upartial12u}
\vertiii{u^{(n)}}
+\vertiii{\partial_x u^{(n)}}
+
\vertiii{\partial^2_x u^{(n)}}
+ 
\||x|^{r}u^{(n)}\|_{L^\infty_TL^2_x}
\leq
C(\rho,T)
\end{equation}
and
\begin{align}\label{um-un}
&
\vertiii{u^{(m)}-u^{(n)}}
+\vertiii{\partial_x (u^{(m)}-u^{(n)})}
+
\vertiii{\partial^2_x (u^{(m)}- u^{(n)})}
+ 
\||x|^{r}(u^{(m)}-u^{(n)})\|_{L^\infty_TL^2_x}\nonumber\\
& \quad \lesssim
\|\varphi_0^{(m)}-\varphi_0^{(n)}\|_{H^2}
+\||x|^{r}
(\varphi_0^{(m)}-\varphi_0^{(n)})\|_{L^2},
\end{align}
where the hidden constant in \eqref{um-un} might depend on $\rho$ and $T$.
\end{Lem}

\begin{proof}[Proof of \eqref{eq:upartial12u}]
Observe that \eqref{eq:psi0(n)} and \eqref{eq:rho(n)} allow us to apply Proposition \ref{Prop:u}
to deduce
\begin{equation}\label{3norm:1}
\vertiii{\varphi^{(n)}}
+
\vertiii{\partial_x\varphi^{(n)}}
+
\vertiii{\partial_x^2 \varphi^{(n)}}
+
\vertiii{\psi^{(n)}}
+
\vertiii{\partial_x\psi^{(n)}}
+
\||x|^{r}\varphi^{(n)}\|_{L^\infty_T L^2_x}
\leq C(\rho, T)
\end{equation}
for certain $T>0$. Next we analyze each of the four terms appearing on the left hand side of \eqref{eq:upartial12u}.\\

To simplify the notation, let's call
\begin{equation}\label{eq:En}
\mathcal{E}^{(n)}(x,t) 
:=\exp\Big(
i\lambda\int_{-\infty}^{x}|\varphi^{(n)}(y,t)|^2dy\Big).
\end{equation}
Since 
$|\mathcal{E}^{(n)}(x,t)|
=1$, 
it is straightforward to see that
\begin{equation}\label{unorm}
\vertiii{u^{(n)}}
+
\||x|^r u^{(n)}\|_{L^{\infty}_T L^2_x}
=
\vertiii{\varphi^{(n)}}
+
\||x|^r\varphi^{(n)}\|_{L^{\infty}_T L^2_x}.
\end{equation}

\quad \\
On the other hand, observe that
\begin{align}\label{derivativeu}
\partial_x u^{(n)}
&=
\mathcal{E}^{(n)}
\Big[i\lambda
|\varphi^{(n)}|^2
\varphi^{(n)} 
+\partial_x\varphi^{(n)}\Big],
\end{align}
which implies, by using
\eqref{eq:Sobolevinclusion},
\eqref{eq:Sobolevproduct}
and 
\eqref{eq:Linfinityembeding},
\begin{align}\label{first_derivative_u}
\vertiii{\partial_x u^{(n)}}
&\lesssim
\|\varphi^{(n)}\|^3_{L^{\infty}_T H^1_x}
+
\|\partial_x\varphi^{(n)}\|_{L^{\infty}_T L^2_x}
+
\|\varphi^{(n)}\|^2_{L^{\infty}_TH^{1}_x}
\|\varphi^{(n)}\|_{L^{4}_TL^{\infty}_x}
+
\|\partial_x\varphi^{(n)}\|_{L^{4}_TL^{\infty}_x} \nonumber \\
&\lesssim
\Big(
\vertiii{\varphi^{(n)}}
+
\vertiii{\partial_x \varphi^{(n)}}
\Big)^3
+
\vertiii{\partial_x\varphi^{(n)}}.
\end{align}

\quad \\
Finally, differentiating once more in \eqref{derivativeu} produces
\begin{align*}
\partial^2_x u^{(n)}
&=
\mathcal{E}^{(n)}
\Big[
-\lambda^2 
|\varphi^{(n)}|^4
\varphi^{(n)}
+
3i\lambda
|\varphi^{(n)}|^2
\partial_x\varphi^{(n)} 
+
i\lambda
(\varphi^{(n)})^2
\partial_x\overline{\varphi^{(n)}}
+
\partial^2_x\varphi^{(n)}\Big].
\end{align*}
Thus,
\begin{align}\label{second_derivative_u}
\vertiii{\partial^2_x u^{(n)}}
&\lesssim 
\|\varphi^{(n)}\|^5_{L^{\infty}_T H^1_x}
+
\|\varphi^{(n)}\|^3_{L^{\infty}_T H^1_x}
+
\|\partial_x^2\varphi^{(n)}\|_{L^{\infty}_T L^2_x} \nonumber \\
& \quad +
\|\varphi^{(n)}\|^4_{L^{\infty}_T H^1_x}
\|\varphi^{(n)}\|_{L^{4}_T L^\infty_x}
+
\|\varphi^{(n)}\|^2_{L^{\infty}_T H^1_x}
\|\partial_x \varphi^{(n)}\|_{L^{4}_T L^\infty_x}
+
\|\partial_x^2\varphi^{(n)}\|_{L^{4}_T L^\infty_x}
\nonumber \\
&\lesssim 
\Big(
\vertiii{\varphi^{(n)}}
+
\vertiii{\partial_x \varphi^{(n)}}
\Big)^5
+
\Big(
\vertiii{\varphi^{(n)}}
+
\vertiii{\partial_x \varphi^{(n)}}
\Big)^3
+
\vertiii{\partial_x^2\varphi^{(n)}}.
\end{align}

\quad \\
Combining  \eqref{unorm},\eqref{first_derivative_u} and \eqref{second_derivative_u} 
with 
\eqref{3norm:1} we deduce 
\eqref{eq:upartial12u}.
\end{proof}

\begin{proof}[Proof of \eqref{um-un}]
Once again, \eqref{eq:psi0(n)} and \eqref{eq:rho(n)} together with Proposition \ref{Prop:u}
yields to
\begin{align}\label{3norm:2}
& \vertiii{\varphi^{(m)}-\varphi^{(n)}}
+\vertiii{\partial_x(\varphi^{(m)}-\varphi^{(n)})}
+\vertiii{\partial_x^2(\varphi^{(m)}-\varphi^{(n)})} \nonumber \\
& \qquad 
+\vertiii{\psi^{(m)}-\psi^{(n)}}
+\vertiii{\partial_x (\psi^{(m)}-\psi^{(n)})}
+ \||x|^{r}(\varphi^{(m)}-\varphi^{(n)})\|_{L^\infty_T L^2_x}
\nonumber\\
&\quad  \lesssim
\|\varphi^{(m)}_0-\varphi^{(n)}_0\|_{H^1}
+\|\psi^{(m)}_0-\psi^{(n)}_0\|_{H^1}
+\||x|^{r}(\varphi^{(m)}_0-\varphi^{(n)}_0)\|_{L^2},
\end{align}
for certain $T>0$. Furthermore, \eqref{eq:34},
\eqref{eq:constrain_psi0}, and \eqref{eq:rho(n)}
give us
\begin{equation}\label{eq:diffpsi}
\|\psi^{(m)}_0-\psi^{(n)}_0\|_{H^1}
\lesssim
\|\varphi^{(m)}_0-\varphi^{(n)}_0\|_{H^2}
+
\||\varphi_0^{(m)}|^2\varphi^{(m)}_0
-|\varphi_0^{(n)}|^2\varphi^{(n)}_0\|_{H^1} 
\lesssim
\|\varphi^{(m)}_0-\varphi^{(n)}_0\|_{H^2}.
\end{equation}

Next we denote each of the four terms on the left hand side of \eqref{um-un} 
by $\textbf{A}$, $\textbf{B}$, $\textbf{C}$, and $\textbf{D}$ respectively, and study them separately below. \\

Recalling the definition of $\mathcal{E}^{(n)}$ introduced in \eqref{eq:En}, we can write
\begin{equation}\label{eq:differenceu's}
u^{(m)}-u^{(n)}
= 
\mathcal{E}^{(m)}
\Big[\varphi^{(m)}-\varphi^{(n)}\Big]
+
\Big[\mathcal{E}^{(m)}-\mathcal{E}^{(n)}\Big]
\varphi^{(n)}.
\end{equation}
Moreover, the mean value theorem gives
\begin{align}\label{Ft_difference_meanvalue}
&|\mathcal{E}^{(m)}(x,t)-\mathcal{E}^{(n)}(x,t)|
\leq
\Big| \int_{-\infty}^{x}
 \Big(|\varphi^{(m)}(y,t)|^2
-
|\varphi^{(n)}(y,t)|^2 \Big) \, dy \Big| \nonumber\\
& \qquad \leq
 \int_{\R}
\Big| 
|\varphi^{(m)}(y,t)|
-
|\varphi^{(n)}(y,t)|\Big|
\Big(|\varphi^{(m)}(y,t)|
+
|\varphi^{(n)}(y,t)|\Big)
 dy \nonumber\\
&\qquad \leq
\|\varphi^{(m)}
-
\varphi^{(n)}\|_{L^\infty_T L^2_x}
\Big(\|\varphi^{(m)}\|_{L^\infty_T L^2_x}
+
\|\varphi^{(n)}\|_{L^\infty_T L^2_x}\Big),
\quad x \in \R, \quad t \in (0,T].
\end{align}

Putting together 
\eqref{3norm:1},
\eqref{eq:differenceu's}, and
\eqref{Ft_difference_meanvalue}
we arrive to
\begin{align}\label{eq:AD}
\textbf{A} + \textbf{D}
& \lesssim 
\vertiii{\varphi^{(m)}-\varphi^{(n)}}
+
\||x|^{r}(\varphi^{(m)}-\varphi^{(n)})\|_{L^\infty_T L^2_x}.
\end{align}

\quad\\
For the treatment of \textbf{B} we notice that,
\begin{align*}
\partial_x u^{(m)}
-
\partial_x u^{(n)} 
& =   
i \lambda \mathcal{E}^{(m)}
\Big[|\varphi^{(m)}|^2 \varphi^{(m)}
-
|\varphi^{(n)}|^2 \varphi^{(n)}\Big]
+
\mathcal{E}^{(m)}
\Big[\partial_x \varphi^{(m)}
-
\partial_x \varphi^{(n)}\Big]
\nonumber \\
& \quad +
\Big[\mathcal{E}^{(m)}
-
\mathcal{E}^{(n)}\Big]
\Big( 
i \lambda 
|\varphi^{(n)}|^2 \varphi^{(n)}
+
\partial_x \varphi^{(n)}\Big).
\end{align*}
Hence, proceeding similarly to \eqref{first_derivative_u}, properties \eqref{eq:34} and \eqref{Ft_difference_meanvalue} in combination with \eqref{3norm:1} produce
\begin{align}\label{mathbfB}
\textbf{B}
&\lesssim
\vertiii{\varphi^{(m)}-\varphi^{(n)}}+\vertiii{\partial_x(\varphi^{(m)}-\varphi^{(n)})}.
\end{align}

\quad\\
On the other hand, the difference of second derivatives can be expressed by means of
\begin{align*}
& \partial_x^2 u^{(m)}
 -
\partial_x^2 u^{(n)} 
 =   
- \lambda^2 
\mathcal{E}^{(m)}
\Big[ 
|\varphi^{(m)}|^4\varphi^{(m)}
-
|\varphi^{(n)}|^4\varphi^{(n)}
\Big] \\
& \quad + 3i \lambda
\mathcal{E}^{(m)}
\Big[ 
|\varphi^{(m)}|^2 \partial_x\varphi^{(m)}
-
|\varphi^{(n)}|^2 \partial_x\varphi^{(n)}
\Big] \\
& \quad +
i \lambda
\mathcal{E}^{(m)}
\Big[ 
(\varphi^{(m)})^2 \partial_x\overline{\varphi^{(m)}}
-
(\varphi^{(n)})^2 \partial_x\overline{\varphi^{(n)}}
\Big] 
+
\mathcal{E}^{(m)}
\Big[ 
\partial_x^2 \varphi^{(m)}
-
\partial_x^2 \varphi^{(n)}
\Big] \\
& \quad +
\Big[ 
\mathcal{E}^{(m)}
-
\mathcal{E}^{(n)}
\Big]
\Big( 
-\lambda^2 
|\varphi^{(n)}|^4
\varphi^{(n)}
+
3i\lambda
|\varphi^{(n)}|^2
\partial_x\varphi^{(n)} 
+
i\lambda
(\varphi^{(n)})^2
\partial_x\overline{\varphi^{(n)}}
+
\partial^2_x\varphi^{(n)}
\Big),
\end{align*}
which as above leads to
\begin{align}\label{mathbfC}
\textbf{C}
&\lesssim
\vertiii{\varphi^{(m)}
-
\varphi^{(n)}}
+\vertiii{\partial_x
(\varphi^{(m)}-\varphi^{(n)})
}
+
\vertiii{\partial^2_x(\varphi^{(m)}- \varphi^{(n)})}.
\end{align}

\quad \\
In summary, the estimates 
\eqref{eq:AD}, \eqref{mathbfB}
and
\eqref{mathbfC} jointly with  
\eqref{3norm:2} and \eqref{eq:diffpsi}
conclude the proof of \eqref{um-un}.
\end{proof}

We are now in position to finish the proof of Theorem \ref{Th:main}. Properties \eqref{um-un} and
\eqref{limitdifference_of_varphi} guarantee that, for some $T>0$ sufficiently small, 
the family $\{u^{(n)}\}_{n \in \N}$ is indeed a Cauchy sequence. If we denote by $u$ its limit, in particular it satisfies
\begin{equation*}
u(\cdot,t)\in H^{2}(\R)\cap L^2(|x|^{2r}dx),
\quad t \in (0,T].    
\end{equation*}
Moreover, recalling that the initial data $\varphi^{(n)}_0$ and $\psi^{(n)}_0$ were taken verifying the constraint \eqref{eq:psi0(n)},
then by Lemma \ref{Prop:u2} such constraint also holds for the 
the solutions
of the NLS system \eqref{trNSE1}, that is,
\begin{equation*}
\psi^{(n)}(\cdot,t)
=\partial_x\varphi^{(n)} (\cdot,t)
+i\frac{\lambda}{2}|\varphi^{(n)}(\cdot,t)|^2\varphi^{(n)}(\cdot,t), \quad t \in (0,T].
\end{equation*}
Therefore, Remark \ref{Lem:phipsiu} gives, for any $n \in \N$,
\begin{equation*}
\left\{
\begin{array}{ll}
i \partial_t u^{(n)}
+ \partial_x^2 u^{(n)}
=i\lambda \partial_x\big(|u^{(n)}|^2u^{(n)}\big), &  x\in \mathbb{R}, \quad 
t \in (0,T], \\
u^{(n)}(x,0)=u^{(n)}_{0}(x), & x \in \R.
\end{array}
\right.
\end{equation*}
Taking the limit as $n \to \infty$ we conclude 
\begin{equation*}
\left\{
\begin{array}{ll}
i \partial_t u
+ \partial_x^2 u
=i\lambda \partial_x\big(|u|^2u\big), &  x\in \mathbb{R}, \quad t \in (0,T], \\
u(x,0)=u_{0}(x), & x \in \R. 
\end{array}
\right. 
\end{equation*}

Finally, observe that the uniqueness of solutions of the DNLS equation \eqref{DNLS} follows from that of the NLS system \eqref{trNSE1}.
\qed


\quad \\

\end{document}